\documentclass[12pt, a4paper]{amsart}   	
\usepackage{amsmath, amssymb, amscd, latexsym, multicol, verbatim, enumerate} 
\usepackage[colorlinks=true, pdfstartview=FitV, linkcolor=blue, citecolor=blue, urlcolor=blue, breaklinks=true]{hyperref}
\usepackage[top=1.15in, bottom=1.15in, left=1.17in, right=1.17in]{geometry}               		
\usepackage{graphicx}				

\usepackage{color}
\definecolor{p}{rgb}{1, 0, 0}
\definecolor{xin}{rgb}{0,1,0}
\definecolor{pp}{rgb}{1, 0, 1}

\newtheorem{theorem}{Theorem}[section]
\newtheorem{corollary}[theorem]{Corollary} 
\newtheorem{lemma}[theorem]{Lemma}
\newtheorem{proposition}[theorem]{Proposition}

\newtheorem{exam}[theorem]{Example}
\newtheorem{example}[theorem]{Example}

\numberwithin{equation}{section}
\newtheorem{definition}[theorem]{Definition}
\newtheorem{remark}[theorem]{Remark}

\newcommand\blfootnote[1]{%
  \begingroup
  \renewcommand\thefootnote{}\footnote{#1}%
  \addtocounter{footnote}{-1}%
  \endgroup
}

\newcommand{\supp}{{\operatorname*{supp}}}
\newcommand{\es}{{\operatorname*{es}}}
\newcommand{\sspan}{{\operatorname*{span}}}
\newcommand{\Es}{{\operatorname*{Es}}}
\newcommand{\R}{\mathbb{R}}
\newcommand{\C}{\mathbb{C}}
\newcommand{\Z}{\mathbb{Z}}

\title[Simplices in Newton-Okounkov bodies and the Gromov width]{Simplices in Newton-Okounkov bodies and the Gromov width of coadjoint orbits}
\author{Xin Fang, Peter Littelmann, Milena Pabiniak}
\address{\newline Mathematisches Institut, Universit\"at zu K\"oln, Cologne, Germany} 
\email{xinfang.math@gmail.com}
\email{peter.littelmann@math.uni-koeln.de}
\email{pabiniak@math.uni-koeln.de}

\begin{document}
\maketitle
\begin{abstract}
 We give a uniform proof for the conjectured Gromov width of rational coadjoint orbits of all compact connected simple Lie groups, 
 by analyzing simplices in Newton-Okounkov bodies.
\end{abstract}

\blfootnote{\textup{2010} \textit{Mathematics Subject Classification}: \textup{ 53D05, 14D06, 14M15, 17B10}} 
\section{Introduction}
Let $\omega_{st}$ be the standard symplectic form on $\R^{2n}$. The non-squeezing theorem of Gromov affirms that a ball $B^{2n}(r)\subset (\R^{2n}, \omega_{st})$ 
cannot be symplectically embedded into $B^2(R)\times \R^{2n-2}\subset  (\R^{2n}, \omega_{st})$ unless $r\leq R$.
This result motivated the quest for the largest ball that could be symplectically embedded into a given symplectic manifold $(M, \omega)$.
The {\it Gromov width} of a $2n$-dimensional symplectic manifold $(M,\omega)$
is the supremum of the set of $a$'s such that the ball of {\it capacity} $a$ (radius $\sqrt{\frac{a}{\pi}}$),
$$ B^{2n}_a = \big \{ (x_1,y_1,\ldots,x_n,y_n) \in  \R^{2n} ; \ \pi \sum_{i=1}^n (x_i^2+y_i^2) < a \big 
\} \subset  (\R^{2n}, \omega_{st}), $$
 can be symplectically
embedded in $(M,\omega)$. 

In this article, we analyze the Gromov width of an important class of symplectic manifolds formed by the orbits of the coadjoint action of compact Lie groups.
Let $K$ be a compact Lie group, and  let $\mathfrak k^*$ be the dual of its Lie algebra $\mathfrak k$. 
Each orbit 
$\mathcal{O}\subset \mathfrak k^*$ of the coadjoint action of $K$ on $\mathfrak k^*$
is naturally equipped with the Kostant-Kirillov-Souriau symplectic form, $\omega^{KKS}$, defined by:
$$\omega^{KKS}_{\xi}(X^\#,Y^\#)=\langle \xi, [X,Y]\rangle,\;\;\;\xi \in \mathcal{O} \subset \mathfrak k^*,\;X,Y \in \mathfrak k,$$
where $X^\#,Y^\#$ are the vector fields on $\mathfrak k^*$ induced by $X,Y \in \mathfrak k$ via the coadjoint action of $K$. 
It has been conjectured that the Gromov width of a coadjoint orbit 
$(\mathcal{O}_{\lambda}, \omega^{KKS})$ of $K$,
through a point $\lambda$ in a positive Weyl chamber, is given by the formula \eqref{gw formula}.
The conjecture had been proved for many, but not all cases (see Section \ref{methods}). Moreover, it was unsatisfactory that the proofs for the lower bounds were different for each group. As the conjectured Gromov width can be expressed by one formula for all compact connected Lie groups, \eqref{gw formula}, one would like to have a uniform proof which works
for all these Lie groups. 
\par
The current article provides such a proof for all coadjoint orbits $\mathcal{O}_{\lambda}$ with $\lambda$ lying on some rational line, for any compact connected simple Lie group. 
We fix an ad-invariant inner product $\langle, \rangle$ on $\mathfrak k$ and use it to identify $\mathfrak k^*$ with $\mathfrak k$.
Recall that a coroot $\alpha^{\vee} \in \mathfrak k$ of a root $\alpha$ is defined to be $ \frac{2\alpha}{\langle \alpha, \alpha \rangle}$.
Our main theorem is the following:
\begin{theorem}\label{theorem lb gw}
Let $K$ be a compact connected simple Lie group.
The Gromov width of a coadjoint orbit $\mathcal{O}_\lambda$ through a point $\lambda$ lying on some rational line in $\mathfrak k^*$, equipped with the Kostant-Kirillov-Souriau symplectic form, is at least
\begin{equation}\label{gw formula}
\min\{\, \left|\left\langle \lambda,\alpha^{\vee} \right\rangle \right|;\  \alpha^{\vee} \textrm{ a  coroot and }\left\langle \lambda,\alpha^{\vee} \right\rangle \neq0\}.
 \end{equation}
\end{theorem}
It has already been proved in \cite{CC} by Caviedes Castro that the Gromov width of $\mathcal{O}_\lambda$ cannot be greater than \eqref{gw formula}. 
Combining these results we immediately obtain:
\begin{corollary}\label{cor gw}
Let $K$ be a compact connected simple Lie group.
The Gromov width of a coadjoint orbit $\mathcal{O}_\lambda$ through a point $\lambda$ lying on some rational line, equipped with the Kostant-Kirillov-Souriau symplectic form, is given by \eqref{gw formula}.
\end{corollary}
\begin{exam}
Let us reformulate the result for $G=SU(n, \C)$. 
As usual, we identify $\mathfrak{su}(n)$
 with the set of $n \times n$ traceless Hermitian matrices, and denote by 
 $\epsilon_i$ the restriction to $\mathfrak{su}(n)$ of a linear map sending a matrix $A$ to $Trace\ (\overline{E}^T_{ii}A)=
 Trace\ (E_{ii}A)$, where $E_{ii}$ is the
 $n \times n$ matrix with $(i,i)$ entry equal to $1$ and all other entries equal to $0$.
Then $\{\epsilon_i-\epsilon_{j};\, i\neq j\}$ form a root system for the  
special unitary group $SU(n, \C)$. 
The above Corollary says that the Gromov width of a coadjoint orbit $\mathcal{O}_\lambda$ of $SU(n, \C)$, 
passing through a point 
$\lambda=\sum_{j=1}^n \lambda_j \epsilon_j \in \mathfrak{su}(n)^*$, $\lambda_1\geq \ldots \geq \lambda_n$, 
 lying on some rational line, is equal to 
$$\min \{|\lambda_i-\lambda_j|;\ i,j \in \{1,\ldots, n\},\ \lambda_i\neq \lambda_j\}.$$

To be more concrete, let us consider $K=SU(3)$, $\lambda= 2\epsilon_1 + \epsilon_2$ and the coadjoint orbit $\mathcal{O}_\lambda$ which, due to integrality of $\lambda$, admits an algebraic interpretation as a flag variety.
Then Theorem~\ref{theorem lb gw} predicts that the Gromov width is $1$. 
In Section \ref{section NObodies}
we use the theory of Newton-Okounkov bodies to get bounds for the Gromov width: the size of a simplex which can be embedded into the Newton Okounkov body associated to $\mathcal{O}_\lambda$ is a lower bound for the Gromov width for $\mathcal{O}_\lambda$ (Theorem~\ref{summary3}).
In this case the Newton-Okounkov body is the convex polytope in $\mathbb R^3$
defined by the equations (see \cite{FFL}):
$$
\{(x_1,x_2,x_3)\in \mathbb R^3; x_1,x_2,x_3\ge 0, x_1\le 1, x_2\le 1, x_1+x_2+x_3\le 2\}.
$$
The standard simplex of size $1$
$$
\Delta:=Conv\{(0,0,0),(1,0,0),(0,1,0),(0,0,1)\}.
$$
embeds into this Newton-Okounkov body.
Corollary~\ref{cor gw} implies 
that the Gromov width is equal to $1$. 
\end{exam} 
As a side result, we also observe that in certain cases the supremum in the definition of the Gromov width is attained.
\begin{proposition}\label{prop gw attained}
Let $K$ be a compact connected simple Lie group, not of type $\tt G_2, \tt F_4$ or $\tt E_8$,
and let $(\mathcal{O}_\lambda, \omega_\lambda^{KKS})$ be its generic coadjoint orbit, through a point $\lambda$ lying on some rational line in $\mathfrak k^*$.
Then there is a symplectic embedding of a ball of capacity  \eqref{gw formula} into $(\mathcal{O}_\lambda, \omega_\lambda^{KKS})$.
\end{proposition}

The paper is organized as follows:  After some comments (Section~\ref{methods})
about the history of the subject and the development of the mathematical tools related to the problem,
we recall (Section~\ref{section NObodies}) in more detail how Newton-Okounkov bodies can be used to
analyze the Gromov width. In Section~\ref{proof}, we give a proof of Theorem~\ref{theorem lb gw} up to a construction
of certain simplices embedded in Newton-Okounkov bodies for coadjoint orbits $\mathcal{O}_\lambda$
associated to integral weights. This is done in Section~\ref{good orderings}, after recalling Lie-theoretic
constructions of Newton-Okounkov bodies (Section~\ref{essmon}).
In Section~\ref{convex ordering} and ~\ref{cominuscule telescopes}  
we provide two alternative constructions of simplices embedded in Newton-Okounkov bodies. 
\vskip 3pt
\noindent
{\bf Acknowledgments.} The work of Xin Fang was partially supported by the Alexander von Humboldt Foundation.
The authors are grateful to the anonymous referees for their comments which improved the quality of this work.

\section{A bit of history and methods for finding lower bounds of the Gromov width}\label{methods}

Recall that every coadjoint orbit intersects a chosen positive Weyl chamber in a single point, providing a bijection between the coadjoint orbits and points in a positive Weyl chamber. Orbits intersecting the interior of a positive Weyl chamber are called {\it generic orbits}. They are of maximal dimension among coadjoint orbits of $K$, and are diffeomorphic to the quotient $K/S$, where $S$ is a maximal torus of $K$. 
Orbits intersecting a positive Weyl chamber at its boundary are called {\it degenerate orbits}.
For example, when $K=U(n,\C)$ is the unitary group, a coadjoint orbit can be identified with the set of Hermitian matrices with a fixed set of eigenvalues. The orbit is generic if all eigenvalues are different, and in this case it is diffeomorphic to the manifold of complete flags in $\C^n$.

Many cases of the conjecture about the Gromov width of coadjoint orbits had already been proved:
\begin{itemize}
\item
Karshon and Tolman in \cite{KT}, and independently Lu in \cite{LuGW}, proved the conjecture for complex Grassmannians (which are degenerate coadjoint orbits of $U(n,\C)$);
\item Zoghi in \cite{Z} proved it for generic indecomposable\footnote{A coadjoint orbit through a point $\lambda$ in the interior of a chosen positive Weyl chamber is called  indecomposable in \cite{Z} if there exists a simple positive root $\alpha$ such that for any positive root $\alpha'$ there exists a positive integer $k$ such that $\langle \lambda, \alpha' \rangle=k \langle \lambda, \alpha \rangle$. }  orbits of $U(n,\C)$;
\item Moreover, he showed that the formula \eqref{gw formula} gives an upper bound for the Gromov width of generic, indecomposable orbits of any compact connected Lie group;
\item In \cite{CC}, Caviedes Castro extended the above result about the upper bound by removing the generic and indecomposable assumptions. This concludes the proof of the upper bound part of the conjecture;
\item The third author showed in \cite{P} that the formula \eqref{gw formula} gives a lower bound of the Gromov width of (not necessarily generic) coadjoint orbits of
 $U(n,\C)$, $SO(2n,\C)$ and $SO(2n+1,\C)$. (The result about $SO(2n+1,\C)$ works only for orbits satisfying one mild technical condition: the point $\lambda$ of intersection of the orbit and a chosen positive Weyl chamber should not belong to a certain subset of one wall of the chamber; see \cite{P} for more details. In particular, all generic orbits satisfy this condition);
 \item Halacheva and the third author in \cite{HP} proved that the lower bound is given by  \eqref{gw formula} for generic orbits of the symplectic group $\text{Sp}(n)=U(n,\mathbb{H})$;
 \item Lane in \cite{L} proved that the lower bound is given by  \eqref{gw formula} for generic orbits of the exceptional group $G_2$;
 \item Additionally, some of the coadjoint orbits fall into the category of manifolds analyzed by Loi and Zuddas, and their Gromov widths are found in \cite{LZ}.
\end{itemize}
In particular, the Gromov width of coadjoint orbits of $E_6$, $E_7$, $E_8$, $F_4$ and of some coadjoint orbits of $SO(2n+1)$, $Sp(n)$ and $G_2$ was not known before we proved Theorem \ref{theorem lb gw}.

We briefly explain how one may prove claims like Theorem \ref{theorem lb gw}.
Lower bounds of the Gromov width are found by providing explicit symplectic embeddings of balls.
If the given manifold $M$ is equipped with an effective Hamiltonian action of a compact torus $S$, such 
embeddings can be constructed by ``flowing along" the flow of the vector fields induced by the action 
(\cite[Proposition 2.8]{KT}) and can be read off from the image of the momentum 
map\footnote{
The momentum map for a Hamiltonian $S$ action on $M$ is a map $\Phi \colon M \rightarrow \mathfrak{s}^*$ such that for any $\xi \in \mathfrak{s}$ the differential of $p \mapsto \Phi(p) (\xi)$ is equal to $\omega(\xi,\_ )$. Thus it is unique only up to adding a constant.
}
$\Phi \colon M \rightarrow \mathfrak{s}^*$ associated to this Hamiltonian action. 
The situation is especially nice if the action is {\it toric}, that is, the dimension of the torus 
is equal to the complex dimension of the manifold.
We describe this basic case more carefully. 
Identify $ \mathfrak{s}^*$ with $ \R^{\dim S}$, thinking of  the circle as $S^1 = \R / \Z$, so that
the lattice of $ \mathfrak{s}^*$ is mapped to $\Z^{\dim S}$ in $\R^{\dim S}$.
The momentum map for the standard $S=(S^1)^n$ action on $(\R^{2n}, \omega_{st})$ maps a ball of capacity $a$ into 
an $n$-dimensional simplex of size $a$, closed on $n$ sides: 
\begin{equation}\label{simplex}
\mathfrak S^n(a):=\{(x_1,\ldots,x_n) \in \R^n;\ 0\leq x_j< a,\  \sum_{j=1}^n x_j< a\}.
\end{equation}
Conversely, suppose that for a toric manifold $(M^{2n},\omega)$ a ``corner" of the image $\Phi(M)$ under a momentum map $\Phi \colon M \rightarrow \mathfrak{s}^*\cong \R^{n}$ is a simplex of size $a$, i.e.
there exist
 $\Psi \in GL(n,\Z)$, $ x \in
\R^n$, and an open affine half-space $H$ such that 
$$\Psi (\mathfrak S^n(a))+\,x =  \Phi(M)\cap H.$$
Then a ball of capacity $a$ can be symplectically embedded in $(M^{2n},\omega)$.
The appearance of $\Psi \in GL(n,\Z)$ arises from the non-canonical identification of $\mathfrak{s}^*$ with $\mathbb{R}^n$ (which depends on a chosen splitting of $S$ into a product of circles), whereas $x$ appears because the momentum map is unique only up to adding a constant.
This result was later generalized to open simplices contained somewhere in the momentum map image (\emph{i.e.}, not necessarily being $\Phi(M)\cap H$).
Then one obtains only embeddings of balls of capacities $a - \varepsilon$ for any $\varepsilon>0$, still implying that the Gromov width is at least $a$.
More precisely:
\begin{proposition}\cite[Proposition 1.3]{LuSC}\cite[Proposition 2.5]{P}\label{embedding}
For any connected, proper (not necessarily compact) Hamiltonian $(S^1)^n$-space $M$ of dimension $2n$, with a momentum map $\Phi$, the Gromov width of $M$ is at least
$$\sup \{a>0\,;\, \exists \; \Psi \in GL(n,\Z), x \in
\R^n,\textrm{ such that }
\Psi (\textrm{int }\mathfrak S^n(a))+\,x \subset \Phi(M) \}.
$$
\end{proposition}

In the case where the action is not toric, one needs to look for projections of simplices not in the whole moment map image, but in its part called a centered region (\cite[Proposition 2.8]{KT}).
 Coadjoint orbits of $K$ are equipped with Hamiltonian (though usually not toric) actions of the maximal torus of $K$. Applying \cite[Proposition 2.8]{KT} to this action, Zoghi in \cite{Z} constructed symplectic embeddings, proving that the Gromov widths of generic indecomposable coadjoint orbits of $U(n)$ are given by \eqref{gw formula}. For non-simply laced groups, the same trick does not give the expected lower bound, but a weaker one (\cite[Appendix A]{Pso}). To obtain the good lower bounds for the coadjoint orbits of $SO(2n)$ and $SO(2n+1)$, (\emph{i.e.}, equal to \eqref{gw formula}), a different action was used: a Gelfand-Tsetlin action (\cite{P}). 
This action is defined only on an open dense subset of the orbit, but there it is toric and therefore provides embeddings of relatively big balls.
This approach fails for the symplectic group. The corresponding Gelfand-Tsetlin action is not toric in that case, and although one still obtains some symplectic embeddings of balls, these balls are of capacities smaller than the expected Gromov width.
\par
A new upgrade in these tools came with the work of Harada and Kaveh \cite{HK}, where the idea of a toric degeneration was brought from algebraic to symplectic geometry. 
A toric degeneration of a complex algebraic variety $X$ is a flat family over $\C$, with generic fibers $X_z$ isomorphic to $X$, and the special fiber $X_0$ being a toric variety. One constructs such a degeneration from a very ample Hermitian line bundle over $X$ and a valuation on its sections (satisfying certain assumptions).
Harada and Kaveh showed how to create a degeneration of a given symplectic manifold (with some relatively mild assumptions)
to a toric variety, keeping track of the symplectic form, and in such a way that the toric action on that variety can be pulled back to a toric action on an open dense subset of the symplectic manifold (Theorem~\ref{summary2}). A toric action obtained in this way was used by Halacheva and the third author in \cite{HP} to prove that the Gromov width of generic coadjoint orbits of the symplectic group $\text{Sp}(n)=U(n,\mathbb{H})$ has a lower bound as in \eqref{gw formula}. 
As this type of argument could be used for any (compact connected and simple) Lie group, it prompted the idea of having one unified proof for all coadjoint orbits.
In the next section, we explain in more detail how this new tool can be used to analyze the Gromov width.

\section{Newton-Okounkov bodies and root subgroups}\label{section NObodies}
As before, let $K$ be a compact connected simple Lie group. 
Since $K$ and its universal cover $\widetilde{K}$ differ only by a finite group which is central in $\widetilde{K}$, both have identical coadjoint orbits. Thus, without loss of generality,
we can assume that $K$ is simply connected. In this section we use
the strong relationship between the theory of coadjoint orbits for a compact Lie group 
and generalized flag varieties in the algebraic setting for its complexification.
We apply tools from algebraic geometry (theory of Newton-Okounkov bodies) and representation theory to the algebraic
setting, which then, following the ideas of Harada and Kaveh \cite{HK}, will lead us
to the proof of the main theorem.

Denote the complexification of $K$ by $G=K_{\mathbb C}$. Let $\mathfrak g$ be the Lie
algebra of $G$.  We fix a maximal torus $S\subset K$.
Its complexification $T=S_{\mathbb C}$ is then a maximal (algebraic) torus in $G$. Denote by $\mathfrak t$
its Lie algebra. Then $\mathfrak t$ is a Cartan subalgebra of $\mathfrak g$.
We fix a Borel subalgebra $\mathfrak b$ of $\mathfrak{g}$ containing $\mathfrak t$. Let $B\subset G$
be the corresponding Borel subgroup, then $T\subset B$.
Let $U^-$ be the unipotent radical of $B^-$, the Borel subgroup opposite to $B$. 
Let $\lambda$ be a dominant integral weight and denote by $\supp(\lambda)$ the {\it support of $\lambda$},
\emph{i.e.,} the set of fundamental weights occurring with a nonzero coefficient in writing $\lambda$ into a sum of fundamental weights. The set of dominant integral weights is denoted by $\Lambda^+$.
For the irreducible representation $V(\lambda)$ of $G$ of highest weight $\lambda$, let 
$\mathbb Cv_\lambda$ be the highest weight line. Let $P=P_{\lambda}\supseteq B$ be the 
normalizer in $G$ of this line. Recall that $P$ depends only on the support $\supp(\lambda)$
of $\lambda$, not on the weight itself. The associated line bundle $\mathcal L_\lambda$ on $G/P$
is very ample, thus, after fixing a Hermitian structure on $\mathcal L_\lambda$, one can equip $G/P$ with a symplectic structure $\omega_{\lambda}$ induced from the Fubini-Study form on the projective space 
$\mathbb P(\mathrm{H}^0(G/P,\mathcal L_\lambda)^*)=\mathbb P(V(\lambda))$ via the Kodaira embedding.
\par
To compare the algebraic and the compact setting, let $K_{P}=K\cap P$.  Then $G/P=K/K_{P}$ can be identified with the highest weight orbit
$G.[v_\lambda]\subset \mathbb P(V(\lambda))$. 
Recall that a dominant integral weight is an integral point in the chosen positive Weyl chamber.
The coadjoint orbit $\mathcal{O}_{\lambda}$ of $K$ through $\lambda$, is diffeomorphic to $K/K_{P}$, and, when equipped with the Kostant-Kirillov-Souriau symplectic form, it is symplectomorphic to $(G/P, \omega_{\lambda})$ (see for example \cite[Remark 5.5]{CC}). 
\par
To construct toric degenerations of $G/P$, we use the theory of Newton-Okounkov bodies \cite{KK,LM}
and the method of birational sequences \cite{FFL}.
Let $U^-_{P}$ be the unipotent radical of the parabolic
subgroup $P^-$, opposite to $P$. The birational map $U^-_{P}\rightarrow G/P$, $u\mapsto u.[\text{id}]$, induces an isomorphism of fields $\mathbb C(U^-_P)\simeq \mathbb C(G/P)$.
\par
For a positive root $\beta$, denote by $\mathfrak g_{-\beta}\subset \text{Lie\,} U^-$ the associated root subspace
and let $U_{-\beta}=\exp \mathfrak g_{-\beta}\subset U^-$ be the corresponding root subgroup.
Let 
\begin{equation}\label{def_phi_p}
\Phi^+_{P}=\{\beta\in\Phi^+;\ \mathfrak g_{-\beta}\subset \text{Lie\,} U^-_{P}\}.
\end{equation} 
We fix an enumeration $\underline{\beta}=\{\beta_1,\ldots,\beta_N\}$ of the roots in $\Phi^+_{P}$.
The product map
$$
\pi: U_{-\beta_1}\times \cdots\times U_{-\beta_N}\rightarrow U^-_{P}
$$
is known to be an isomorphism of affine varieties. We write $\mathbb C[x_{\beta_i}]$ for the
coordinate ring of $U_{-\beta_i}$, which is, as affine variety, just an affine line.

Let $>_{\mathrm{r}}$ be the right lexicographic order on $\mathbb N^N$, that is for two tuples $\underline{\mathbf m}=(m_1,m_2,\ldots,m_N)$ and $\underline{\mathbf k}=(k_1,k_2,\ldots,k_N)$, we say that
$\underline{\mathbf m}>_{\mathrm{r}} \underline{\mathbf k}$, if there exists $1\leq s\leq N$ such that $m_N=k_N,\cdots, m_{s+1}=k_{s+1}$, and $m_{s}>k_{s}$.
We get an induced monomial order on the set of monomials in $\mathbb C[x_{\beta_1},\ldots,x_{\beta_N}]$ by defining
$x^{\underline{\mathbf m}}>_{\mathrm{r}} x^{\underline{\mathbf k}}$ if 
${\underline{\mathbf m}}>_{\mathrm{r}} {\underline{\mathbf k}}$, where for $\underline{\mathbf t}=(t_1,\ldots,t_N)$, $x^{\underline{\mathbf t}}:=x_{\beta_1}^{t_1}\cdots x_{\beta_N}^{t_N}$.
We define an induced $\mathbb Z^N$-valued valuation on $\mathbb C(G/P)=\mathbb C(x_{\beta_1},\ldots,x_{\beta_N})$ in the following way: for a nonzero polynomial in $\mathbb C[x_{\beta_1},\ldots,x_{\beta_N}]$,
$$
\nu(\sum a_{\underline{\mathbf m}} x^{\underline{\mathbf m}}):=\min\{\underline{\mathbf m}; a_{\underline{\mathbf m}}\not=0\}
$$
 and $\nu(\frac{f}{g})=\nu(f)-\nu(g)$ for a nonzero
element $f/g\in \mathbb C(G/P)$.

Let $R_\lambda$ be the ring of sections
$$
R_\lambda=\bigoplus_{\ell\geq 0}\, \mathrm{H}^0(G/P,\mathcal L_\lambda^{\otimes \ell}).
$$
We fix a nonzero highest weight section $s_0\in \mathrm{H}^0(G/P,\mathcal L_\lambda)$ and form a graded monoid (\emph{i.e.}, a graded semigroup with identity)
\begin{align*}
\Gamma_{\lambda,\,\underline{\beta}}&=\bigcup_{\ell\in\mathbb N} \Gamma_{\lambda,\,\underline{\beta}}(\ell) \subset \mathbb N\times \mathbb Z^N,\\ \ \text{where}\ \
\Gamma_{\lambda,\,\underline{\beta}}(\ell)&=\{(\ell,\nu(s/s_0^\ell)); s\in \mathrm{H}^0(G/P,\mathcal L_\lambda^{\otimes \ell})\}.
\end{align*}
The Newton-Okounkov body $\Delta_\lambda(\underline{\beta})$ associated to the valuation is the convex body defined as
$$
\Delta_\lambda(\underline{\beta})=\overline{ Conv\ \big\{ \tfrac{1}{\ell} \underline{\mathbf m}; (\ell,\underline{\mathbf m})\in \Gamma_{\lambda,\,\underline{\beta}}\big\}}\subset \mathbb R^N.
$$
Equivalently, $\Delta_\lambda(\underline{\beta})$ is the intersection of the closure of the convex hull of $\Gamma_{\lambda,\,\underline{\beta}} \cup \{(0,0)\} \subset  \mathbb N\times \mathbb Z^N$ with $\{1\} \times \mathbb Z^N$.
The following theorems are results of Anderson \cite{A}, Harada and Kaveh \cite{HK}. 
They hold in a much more general situation, but here we rephrase them according to the special circumstance of this article.
\begin{theorem}[\cite{A}]\label{summary}
If the monoid $\Gamma_{\lambda,\,\underline{\beta}}$ is finitely generated, there exists a flat family $\pi:\frak X\rightarrow \mathbb C$ such that 
for any $z\in\mathbb C\backslash\{0\}$, the fibre $X_z=\pi^{-1}(z)$ is isomorphic to $G/P$, and 
$X_0=\pi^{-1}(0)$ is isomorphic to $\mathrm{Proj\,}\mathbb C[\Gamma_{\lambda,\,\underline{\beta}}]$.  The variety $X_0$ is equipped
with an action of the torus $(\mathbb C^*)^N$. 
The normalization of the variety $X_0$ is the toric variety $X_{\Delta_\lambda(\underline{\beta})}$ associated to the rational polytope $\Delta_\lambda(\underline{\beta})$.
\end{theorem}
Moreover, the torus action on $X_0$ induces a torus action on a subset of $G/P$.
\begin{theorem}[\cite{HK}]\label{summary2}
Assume that $\Gamma_{\lambda,\,\underline{\beta}}$ is finitely generated.
There exists an integrable system $\mu= (\mathcal F_1,\ldots,\mathcal F_N): G/P \rightarrow\mathbb R^N$ on
$(G/P, \omega_{\lambda})$, and the image of $\mu$ coincides with the Newton-Okounkov body $\Delta_\lambda(\underline{\beta})$.
The integrable system generates a torus action on the inverse
image under $\mu$ of the interior of $\Delta_\lambda(\underline{\beta})$, and there the restriction of $\mu$ is a momentum map.
\end{theorem}
In fact, Harada and Kaveh proved that the torus action is defined on a set bigger than just the inverse
image under $\mu$ of the interior of $\Delta_\lambda(\underline{\beta})$.
This is not relevant for proving the main result of this paper, Theorem \ref{theorem lb gw}, but will be needed for proving our side result, Proposition \ref{prop gw attained}.
Recall that a polytope $\Delta \in \R^n$ is called {\it smooth} if there are exactly $n$ edges meeting at each vertex of $\Delta$ and their primitive generators form a $\Z$ basis of $\Z^n$ (see for example \cite[Definition 2.4.2]{CLS}). A point $x$ in the interior of a facet $F$ of a polytope $\Delta$ is called smooth, if $F$ itself is a smooth polytope.
 Note that for a simplex all points are smooth.
\begin{corollary}\label{cor action on smooth pts}
Assume that $\Gamma_{\lambda,\,\underline{\beta}}$ is finitely generated.
The integrable system from Theorem \ref{summary2} generates a torus action on the preimage of the smooth points of $\Delta_\lambda(\underline{\beta})$. 
Therefore, if there exist
 $\Psi \in GL(n,\Z)$, $ x \in
\R^n$, and an open affine half-space $H$ such that 
$$\Psi (\mathfrak S^n(a))+\,x =  \Delta_\lambda(\underline{\beta})\cap H,$$
then a ball of capacity $a$ can be symplectically embedded in $(G/P, \omega_{\lambda})$.
\end{corollary}
\begin{proof}
 The first claim is proved in \cite{HK}, though not stated explicitly, and the second is an immediate corollary.
The integrable system on $G/P$ is induced from the integrable system on the toric variety $X_0$, using a surjective continuous map $\phi \colon G/P \rightarrow X_0$ (see the proof of Theorem 2.19 of \cite{HK}). The map $\phi$ is a symplectomorphism when restricted to $\phi^{-1}(U_0)$, where $U_0$ is the smooth locus of $X_0$ (Corollary 2.10 of \cite{HK}), and the toric action is defined on $\phi^{-1}(U_0)$. 
As the normalization map $X_{\Delta_\lambda(\underline{\beta})}\rightarrow X_0$ induces a bijection on smooth points, the claim follows.
\end{proof}
To show the finite generation of the monoid $\Gamma_{\lambda,\,\underline{\beta}}$ is usually (not only in the flag variety case) a difficult problem. 
Some special cases are presented in Section~\ref{essmon}.

The Newton-Okounkov body may provide interesting information even if $\Gamma_{\lambda,\,\underline{\beta}}$ is \emph{not necessarily finitely generated}.
Kaveh in  \cite[Corollary 12.3 and 12.4]{K} showed that even then one can still form a family $\pi:\frak X\rightarrow \mathbb C$ such that $X_z=\pi^{-1}(z)$ is isomorphic to $G/P$ for $z\neq 0$ and 
$X_0=\pi^{-1}(0)$ is isomorphic to $(\mathbb C^*)^N$.
Recall that by $\textrm{int}\,\mathfrak S^N(r)$ we denote the interior of the $N$-dimensional simplex of size $r$.
\begin{theorem}[\cite{K}]\label{summary3}
The Gromov width of $(G/P,\omega_\lambda)$ is at least R, where R is the supremum of
the sizes of open simplices  that fit (up to $GL(N,\Z)$ transformation) in the interior of the Newton-Okounkov body $\Delta_\lambda (\underline{\beta})$, \emph{i.e.},
$$R= \sup \{r>0\,;\, \exists \; \Psi \in GL(N,\Z), x \in
\R^N\textrm{ so that }\Psi (\textrm{int }\mathfrak S^N(r))+\,x \subset \Delta_\lambda (\underline{\beta}) \}.$$
\end{theorem}

\begin{remark}\rm
The study of embeddings of simplices in Newton-Okounkov bodies turns out to be useful also for other purposes.
For example, see the work of K\"uronya and Lozovanu \cite{KL1,KL2}, who study the positivity for divisors in terms of convex geometry, 
or the work of Ito \cite{I} (see also \cite{N}), who studies Seshadri constants.
\end{remark}

We finish this section with a remark about a relation of our result to the global Seshadri constant. 
Recall that for a complex manifold $(M,J)$ and an ample line bundle $\mathcal{L}$ over $M$ the {\it Seshadri constant} of $\mathcal{L}$ at a point $p \in M$ is defined as the nonnegative real number 
$$\epsilon (M,\mathcal{L},p):=\inf_C \frac{\int_C c_1(\mathcal{L})}{mult_pC},$$
where the infimum is taken over all irreducible holomorphic curves C passing through the point $p$, and $mult_pC$ is the multiplicity of $C$ at $p$. The global Seshadri constant is defined as $\epsilon(M,\mathcal{L}):= \inf_{p \in M}\,\epsilon (M,\mathcal{L},p).$ The cohomology class $c_1(\mathcal{L})$ can be represented by a $J$-compatible K\"ahler form $\omega_\mathcal{L}$, 
and the symplectomorphism type of $(M,\omega_\mathcal{L})$ depends only on the cohomology class $[c_1(\mathcal{L})]=[\omega_\mathcal{L}]$.
Biran and Cieliebak in \cite[Proposition 6.3]{BC} showed that 
$$\epsilon(M,\mathcal{L}) \leq \textrm{Gromov width of }(M,\omega_\mathcal{L}).$$
Moreover, \cite[Theorem 1.2 and Lemma 4.4]{I} of Ito (see also \cite{N}) imply that if $\overline{\mathfrak S^n(a)}$ is a subset of the Newton-Okounkov body associated to the line bundle $\mathcal{L}$ and a lowest term valuation for some coordinate system on $M$, then $\epsilon(M,\mathcal{L})$ is at least the Seshadri constant of the toric variety and the line bundle associated to the polytope $\overline{\mathfrak S^n(a)}$, i.e. of the complex projective space with $\frac{a}{\pi}\omega_{FS}$ (whose Seshadri constant is $a$). For this conclusion of Ito it is not necessary to assume that the associated monoid is finitely generated. Therefore Corollary \ref{cor gw} implies the following.
\begin{corollary}
Let $K$ be a compact connected simple Lie group.
The global Seshadri constant of a coadjoint orbit $\mathcal{O}_\lambda$ through an integral $\lambda$, with the line bundle $\mathfrak{\mathcal{L}}_{\lambda}$, is given by \eqref{gw formula}.
\end{corollary}

\section{The proof of the main result.}\label{proof}
With this background reviewed, we are now ready to present the proof of Theorem \ref{theorem lb gw}, up to 
the detailed analysis of the Newton-Okounkov bodies, which is postponed to the next sections. 
\begin{proof}[Proof of Theorem \ref{theorem lb gw}]
We first prove Theorem \ref{theorem lb gw} for integral $\lambda$ in a positive Weyl chamber, \emph{i.e.}, for a dominant integral weight.
Then the coadjoint orbit $(\mathcal{O}_{\lambda}, \omega^{KKS})$ is symplectomorphic to a flag manifold $(G/P, \omega_{\lambda})$ with a symplectic structure pulled back via Kodaira embedding $G/P\hookrightarrow \mathbb P(\mathrm{H}^0(G/P,\mathcal L_\lambda)^*)$.

In the previous section, we explained how to associate a valuation on the space of sections of $\mathcal L_\lambda$ (and of its tensor products) to a given enumeration $\underline{\beta}=\{\beta_1,\ldots,\beta_N\}$ of roots in $\Phi^+_{P_{}}$ (see \eqref{def_phi_p} for the definition of $\Phi^+_P$). 
From there one obtains a Newton-Okounkov body $\Delta_{\lambda}(\underline{\beta})$.
In Section \ref{good orderings} we fix an enumeration $\underline{\beta}=\{\beta_1,\ldots,\beta_N\}$
of the positive roots in question and analyze the associated Newton-Okounkov bodies $\Delta_{\lambda}(\underline{\beta})$.
In Theorem \ref{Prop:goodordering} we show that the body contains an open simplex of  the size precisely as in \eqref{gw formula}.
Then Theorem \ref{summary3} gives that the Gromov width of $(G/P, \omega_{\lambda})$ is greater or equal \eqref{gw formula}, what proves Theorem \ref{theorem lb gw} for an integral $\lambda$. (See Section~\ref{convex ordering} and~\ref{cominuscule telescopes} for other enumerations.)

If $\lambda \neq 0$ is on a rational line, i.e. there exists $\ell \in \R \setminus \{0\}$ such that $\ell \lambda$ is integral, then 
$(\mathcal{O}_{\lambda}, \ell \omega^{KKS})$ is symplectomorphic to $(\mathcal{O}_{\ell\lambda},  \omega^{KKS})$
and thus 
the Gromov width of $(\mathcal{O}_{\lambda},  \omega^{KKS})$ is $\frac{1}{\ell}$ of the 
Gromov width of $(\mathcal{O}_{\ell\lambda},  \omega^{KKS})$, \emph{i.e.,}
\begin{eqnarray*}
& &\frac{1}{\ell} \min\{\, \left|\left\langle \alpha^{\vee},\ell\lambda \right\rangle \right|;\  \alpha^{\vee} \textrm{ a  coroot and }\left\langle \alpha^{\vee},\lambda \right\rangle \neq0\}\\
&=& \min\{\, \left|\left\langle \alpha^{\vee},\lambda \right\rangle \right|;\  \alpha^{\vee} \textrm{ a  coroot and }\left\langle \alpha^{\vee},\lambda \right\rangle \neq0\}.
\end{eqnarray*}
\end{proof}
\begin{remark}
Of course one would like to extend this result to orbits $\mathcal{O}_{\lambda}$ with arbitrary $\lambda$. Such extension would immediately follow from upper semi-continuity
\footnote{Adjusting a ``Moser type" argument from  \cite{MP} one can show that the Gromov width of $\mathcal{O}_{\lambda}$ is lower semi-continuous.} 
of the Gromov width of $\mathcal{O}_{\lambda}$
 as a function of $\lambda$.
Paul Biran once claimed that he expects all obstructions to embeddings of balls to come from $J$-holomorphic curves. If this expectation turns out to be true, one could prove the upper semi-continuity.
Note that an important implication of the above conjecture of Biran is that the Gromov width of integral symplectic manifolds must be greater or equal to one. This statement was proved, under certain assumption, by Kaveh in \cite{K}.
\end{remark}
\begin{proof}[Proof of Proposition \ref{prop gw attained}]
To prove Proposition \ref{prop gw attained} one repeats the proof of Theorem \ref{theorem lb gw} but using enumerations presented in Section~\ref{cominuscule telescopes}. 
The result follows from Theorem \ref{Thm:cominuscule} and Corollary \ref{cor action on smooth pts}.
\end{proof}


\section{Essential monomials}\label{essmon}
Let $\lambda$ be a dominant integral weight and let $P=P_\lambda$ be the associated parabolic subgroup. Recall that 
$P$ is by definition a standard parabolic subgroup, i.e. $B\subseteq P$.
In \cite{FFL}, the first two authors 
give a representa\-tion-theoretic construction of the monoid
$\Gamma_{\lambda, \underline{\beta}}$ coming from an enumeration $\underline{\beta}=\{\beta_1,\ldots,\beta_N\}$ of roots in $\Phi^+_P$
(see \eqref{def_phi_p} for the definition of $\Phi^+_P$).
For the rest of this section, we fix the enumeration $\underline{\beta}$ and remove it from the notations for simplicity.
 
\subsection{Filtration arising from birational sequences}\label{afiltration} 
Let $\mathfrak n^-_P$ be the Lie algebra of $U^-_P$. For each positive root $\beta$, let $F_\beta$ be a generator of $\mathfrak g_{-\beta}$, $E_\beta$ be a generator of $\mathfrak g_{\beta}$ such that the sub-algebra generated by $E_\beta$, $F_\beta$ and $\beta^\vee$ is isomorphic to $\mathfrak{sl}_2$.
The vectors $\{F_\beta; \beta\in\Phi^+_P\}$
form a vector space basis of $\mathfrak n^-_{P}$. As a vector space, the enveloping algebra $U(\mathfrak n^-_{P})$ admits a PBW-basis (\emph{i.e.}, ordered monomials in the
root vectors):
$$
\{ F^{\underline{\mathbf m}}=F_{\beta_1}^{m_1}\cdots F_{\beta_N}^{m_N}; \underline{\mathbf m}=(m_1,m_2,\ldots,m_N)\in \mathbb N^N\}.
$$
Let $>_{\mathrm{or}}$ be the right opposite lexicographic order on $\mathbb N^N$: that is to say, $\underline{\mathbf m}>_{\mathrm{or}} \underline{\mathbf k}$ if and only if $\underline{\mathbf m}<_{\mathrm{r}} \underline{\mathbf k}$.
We use this total order to define for $\underline{\mathbf m}\in \mathbb N^N$ subspaces of $U(\mathfrak n^-_P)$ as follows:
$$
U(\mathfrak n^-_P)_{<_{\mathrm{or}} \underline{\mathbf m}}=\sspan_{\mathbb{C}}\{ F^{\underline{\mathbf k}};\underline{\mathbf k}<_{\mathrm{or}} \underline{\mathbf m}\},
\quad
U(\mathfrak n^-_P)_{\le_{\mathrm{or}} \underline{\mathbf m}}=\sspan_{\mathbb{C}}\{ F^{\underline{\mathbf k}};\underline{\mathbf k}\le_{\mathrm{or}} \underline{\mathbf m}\}.
$$
For an irreducible representation $V(\lambda)$, we have induced subspaces of $V(\lambda)$:
\begin{equation}\label{filt1}
V(\lambda)_{<_{\mathrm{or}} \underline{\mathbf m}}=(U(\mathfrak n^-_P)_{<_{\mathrm{or}} \underline{\mathbf m}})\cdot v_\lambda,\quad
V(\lambda)_{\le_{\mathrm{or}} \underline{\mathbf m}}=(U(\mathfrak n^-_P)_{\le_{\mathrm{or}} \underline{\mathbf m}})\cdot v_\lambda.
\end{equation}
The subquotient $V(\lambda)_{\le_{\mathrm{or}} \underline{\mathbf m}}/V(\lambda)_{<_{\mathrm{or}} \underline{\mathbf m}}$
 is obviously at most one dimensional.
\begin{definition}\rm
A tuple $\underline{\mathbf m}\in \mathbb N^N$, the monomial $F^{\underline{\mathbf m}}$ and the vector $F^{\underline{\mathbf m}} v_\lambda$ 
are called {\it essential} for $V(\lambda)$, if the subquotient
$V(\lambda)_{\le_{\mathrm{or}} \underline{\mathbf m}}/V(\lambda)_{<_{\mathrm{or}} \underline{\mathbf m}}$ is of dimension one.
\end{definition}

Denote the set of all essential tuples
for $V(\lambda)$ by
$$
\es_{P}(\lambda)=\{\underline{\mathbf m}\in\mathbb N^N; \underline{\mathbf m}\text{\ is essential for\ }V(\lambda)\}.
$$

\subsection{Essential monoids and global version}\label{anothermonoid}

It has been shown in \cite[Proposition 1]{FFL} that for integers $\ell,k\geq 1$, 
$$\{\ell\}\times\es_{P}(\ell\lambda)+\{k\}\times\es_{P}(k\lambda)\subset \{\ell+k\}\times\es_{P}((\ell+k)\lambda)$$ 
(here $+$ stands for the Minkowski sum of two sets), therefore the set
$$
\Es_{P}(\lambda)=\bigcup_{\ell\in\mathbb N}\,\{\ell\}\times\es_{P}(\ell\lambda)\subset \mathbb N\times\mathbb N^N
$$ 
is naturally endowed with the structure of a submonoid of $\mathbb N\times\mathbb N^N$.
Moreover: 
\begin{theorem}[\cite{FFL}]\label{gammagleichess} 
The graded submonoids $\Gamma_\lambda$ and $\Es_{P}(\lambda)$ of $\mathbb N\times\mathbb N^N$ coincide.
\end{theorem}

Let $\mu$ be a dominant integral weight such that 
$\supp(\mu)\subseteq \supp(\lambda)$. In this case $V(\mu)$ is still a cyclic $U(\mathfrak n^-_P)$-module,
so the filtration described in \eqref{filt1} for $V(\lambda)$ also makes sense for $V(\mu)$, and so does the notation
of an essential tuple. We set 
$$
\Es_{P}=\{(\mu,\underline{\mathbf m})\in \Lambda^+\times\mathbb N^N; 
\supp(\mu)\subseteq\supp(\lambda),\,\underline{\mathbf m}\text{\ is essential for\ }V(\mu)\}.
$$

\begin{theorem}[\cite{FFL}]\label{global} 
The set  $\Es_{P}\subset \Lambda^+\times\mathbb N^N$ is a submonoid. In particular, for $(\mu,\underline{\mathbf m}),(\nu,\underline{\mathbf k})\in \Es_{P}$,
one has $\underline{\mathbf m}+\underline{\mathbf k} \in \es_{P}(\mu+\nu)$.
\end{theorem}

Consider the following class of examples. Let $\{\alpha_1,\ldots,\alpha_n\}$ be the set of simple roots for $\mathfrak g$,
$L\supseteq T$ be the Levi subgroup of $P$ and $w_L$ be the longest word in the Weyl group of $L$. 
Fix a reduced decomposition of the longest word $w_0$ in the Weyl group $W$ of $G$, which is of the form $w_0=w_Ls_{i_1}\cdots s_{i_N}$ or $w_0=s_{i_1}\cdots s_{i_N}w_L$:
in the first case we set
\begin{equation}\label{decomplist}
\Phi^+_P=\{\beta_1=w_L(\alpha_{i_1}),\beta_2= w_Ls_{i_1}(\alpha_{i_2}), \ldots,\beta_N=w_Ls_{i_1}\cdots s_{i_{N-1}}(\alpha_{i_N})\};
\end{equation}
while in the second case we set
\begin{equation}\label{decomplistb}
\Phi^+_P=\{\beta_N=w_L(\alpha_{i_N}),\beta_{N-1}=w_Ls_{i_N}(\alpha_{i_{N-1}}),\ldots, \beta_1=w_L s_{i_{N}}\cdots s_2(\alpha_{i_1})\}.
\end{equation}
An enumeration as in  \eqref{decomplist}  or  \eqref{decomplistb} is said to be \emph{induced by a reduced decomposition}.
\begin{theorem}\cite[Corollary 6]{FFL} 
If the enumeration of the positive roots is induced by a reduced decomposition,
then $\Gamma_{\lambda,\,\underline{\beta}}$ is finitely generated and saturated.
In particular, the limit toric variety $X_0$ (see Theorem~\ref{summary}) is normal.
\end{theorem}
\begin{remark}\label{ExplicitNO} 
Even if $\Gamma_{\lambda,\,\underline{\beta}}$ is finitely generated, it  is not an easy task
to give an explicit description of the associated Newton-Okounkov body.  
If the enumeration of the positive roots is induced by a reduced decomposition, using \cite{M}, one can show that these
bodies are related to string polytopes. The string polytopes are described in \cite{BZ,Li},
and further examples are given in \cite{AB} and \cite{FFL}.  
Comments regarding geometric properties of the limit toric variety $X_0$ (see Theorem~\ref{summary}) like Fano and Gorenstein can be found in  \cite{AB}.
\end{remark}
The identification of $\Gamma_\lambda$ with $\Es_{P}(\lambda)$ can be used to construct simplices contained in $\Delta_\lambda$. In the next two sections, we describe three different methods of doing so. 

\section{A special simplex in $\Delta_\lambda$ from good orderings}\label{good orderings}
Recall the usual partial order on positive roots:
$\beta\succ \gamma$ if $\beta-\gamma$ can be written as a sum of positive roots.
In this section, we assume that the enumeration $\underline{\beta}=\{\beta_1,\ldots,\beta_N\}$ of  the roots in $\Phi^+_P $
is a {\it good ordering} in the sense of \cite{FFL}, \emph{i.e.,} if $\beta_i\succ \beta_j$, then $i>j$. Again, as $\underline{\beta} $ is fixed throughout the 
section, we suppress it from the notation. 
For $1\leq k\leq N$, we denote by $\mathbf{e}_k$ the standard coordinate of $\mathbb{R}^N$ whose $k$-th entry is $1$ and other entries are $0$.

\begin{lemma}\label{hilfslemma1}
If $\varpi$  is a fundamental weight contained in $\supp(\lambda)$ and $\beta\in \Phi^+_P$ is such that
$\langle\varpi,\beta^\vee\rangle\not=0$, then the root vector $F_\beta$ is essential for $V(\varpi)$.
\end{lemma}

\begin{proof}
Let $i_0$ be such that $\beta=\beta_{i_0}$ in the enumeration of the roots in $\Phi^+_P$.
Let $F^{\underline{\mathbf k}}$ be an element of the PBW-basis of weight $\beta$. Then either $\underline{\mathbf k}=\mathbf{e}_{i_0}$,
or one has $\beta =\sum_{j=1}^N k_j\beta_j$ with at least two nonzero coefficients. It follows that if $k_j\not=0$, then $\beta\succ\beta_j$,
and hence $\underline{\mathbf k}>_{\mathrm{or}} \mathbf{e}_{i_0}$ in the opposite right lexicographic ordering. In addition,
$\langle\varpi,\beta^\vee\rangle\not=0$ implies $F_\beta v_\varpi\not=0$, which proves that 
$F_\beta v_\varpi$ and hence $F_\beta$ is essential for $V(\varpi)$.
\end{proof}

Recall the definition of $\mathfrak S^N(a)$  in \eqref{simplex}. 
Let $\rho_{P}$ be  the sum of all fundamental weights in $\supp(\lambda)$.

\begin{theorem}\label{Prop:goodordering}
For $k=\min\{\vert \langle\lambda,\alpha^\vee\rangle\vert ; \alpha^{\vee} \textrm{ a  coroot and }\left\langle \lambda,\alpha^{\vee} \right\rangle \neq0\}$,   
one has $\mathfrak S^N(k)\subset \Delta_\lambda$.
\end{theorem}

\begin{proof}
It is easy to see that $k=\min\{\langle\lambda,\beta^\vee\rangle; \beta\in \Phi_P^+\}$, and it
is the maximal integer such that $\lambda=k\rho_{P}+\nu$, where $\nu$ is a dominant integral weight with $\supp(\nu)\subset \supp(\lambda)$.
\par
Let $\beta_i\in \Phi^+_P$. Since $\langle\lambda,\beta_i^\vee\rangle\not=0$, one has $\langle\varpi,\beta_i^\vee\rangle\not=0$
for some fundamental weight $\varpi\in \supp(\lambda)$. By Lemma \ref{hilfslemma1}, $\mathbf{e}_i$ is essential for $V(\varpi)$ for some $\varpi\in \supp(\lambda)$. Since the zero vector is always among the essential tuples for any dominant weight, Theorem~\ref{global} implies that $\mathbf{e}_1,\ldots,\mathbf{e}_N$ are essential for $V(\rho_P)$. The same reasoning implies that  $k\mathbf{e}_1,\ldots,k\mathbf{e}_N$ are essential for $V(k\rho_{P})$, and, again for the same reason, they are essential for $V(\lambda)$. By the convexity of $\Delta_\lambda$, one has $\mathfrak S^N(k)\subset \Delta_\lambda$.
\end{proof}

\begin{remark}\label{remarkpeter}
If $G$ is of type $\tt A_n,C_n, G_2$ or $\tt D_4$, then it
is known, that there exists a good ordering such that the monoid $\Gamma_\lambda=\Es_P(\lambda)$ is 
finitely generated for any dominant integral weight $\lambda$ \cite{FFL1,G,G2}. We conjecture that this holds for any simply connected simple 
complex algebraic group $G$.
\end{remark}
\begin{remark}\label{remarkpeter2}
Let $G$ be of one of the types above. It has been shown that in these cases, there exists a good ordering such that for any dominant weight
the monoid $\Gamma_\lambda=\Es_P(\lambda)$ is finitely generated, saturated and has the Minkowski property \cite{FFL1,G,G2}, i.e., 
$$\{k\}\times\es_{P}(k\lambda)+\{\ell\}\times\es_{P}(\ell\lambda)=\{k+\ell\}\times\es_{P}((k+\ell)\lambda).$$
In particular,
the Newton-Okounkov body $\Delta_\lambda$ is an integral polytope and the limit variety $X_0$ (Theorem~\ref{summary}) is 
a normal toric variety. In addition, if $\mathcal L_\lambda = \mathcal O(-K_{G/P})$ is the anticanonical line bundle, then, by  \cite[Theorem 3.8]{AB},
$X_0$ is a Fano variety and the Newton-Okounkov body is reflexive.
\end{remark}

\section{Additional special simplices in $\Delta_\lambda$}\label{mooreorderings}
Although the first method (Section~\ref{good orderings}) gives a proof of Theorem \ref{theorem lb gw} in full generality, 
we present two additional approaches for generic orbits. In these approaches we use enumerations
of the positive roots ${\underline{\beta}}$ induced by reduced decompositions of $w_0$, so the 
monoids $\Gamma_{\lambda,\,\underline{\beta}}$ are finitely generated and saturated.  It follows that the limit varieties $X_0$ 
(Theorem~\ref{summary}) exist and are normal toric varieties. Moreover, in many of these cases, an explicit
description of the Newton-Okounkov body is known (see \cite{BZ,L}, Remark~\ref{ExplicitNO}).
Therefore, we hope that these approaches could help to solve other related problems about coadjoint orbits
(for example, symplectic packing, formula for potential functions and finding non-displaceable Lagrangians \cite{NNU},  \emph{etc.}). 
Additionally, the degeneration discussed in Section~\ref{cominuscule telescopes} has the advantage that the simplex of appropriate size
is a ``corner" of the Newton-Okounkov body. Therefore, Corollary \ref{cor action on smooth pts} implies that there exists 
an embedding of a ball of capacity given by \eqref{gw formula}, i.e. the supremum appearing in the Gromov width definition is attained.
Moreover, the construction of cominuscule telescopes (Section~\ref{cominuscule telescopes}) may relate to  Thimm's trick \cite{T}.

We  assume in the following that $\lambda$ is a {\it  regular dominant integral weight}, \emph{i.e.}, for any simple root
$\alpha$: $\langle\lambda,\alpha^\vee\rangle>0$. In this case,
$\supp(\lambda)$ is the set of all fundamental weights, so $P=B$ and we omit this subscript to simplify the notation, and
\eqref{gw formula} is equal to $\min\{\langle\lambda,\beta^\vee\rangle;\beta\in \Phi^+\}$.
\subsection{A special simplex in $\Delta_\lambda$ from convex ordering}\label{convex ordering}
Let $w_0=s_{i_1}\cdots s_{i_N}$ be a reduced decomposition of the longest word and let $\underline{\beta}=\{\beta_1,\ldots,\beta_N\}$
be the induced enumeration of all positive roots as in \eqref{decomplistb} with $L=T$.
Given $\underline{m}=(m_1, \ldots, m_N)\in \mathbb{N}^N$, let $F^{\underline{m}}v_{\lambda}$ denote 
$F_{\beta_{1}}^{m_{1}}\cdots F_{\beta_N}^{m_N}v_\lambda$. If $F^{\underline{m}}v_{\lambda}\neq 0$, then it is a $T$-eigenvector of weight $\lambda -m_1\beta_1-\ldots - m_N \beta_N$.
\par
We define a tuple $\underline{\mathbf m}^{\max}=(m_1^{\max},\ldots,m_N^{\max})\in\mathbb N^N$ by descending induction: 
\begin{itemize}
\item $m_N^{\max}$ is the maximal integer with the property
 $F_{\beta_N}^{m_N^{\max}}v_\lambda\not=0$;
\item if $m^{\max}_\ell$ is defined for $\ell=k+1,\ldots, N$, then $m_k^{\max}$
 is defined as  the maximal integer with the property 
$F_{\beta_{k}}^{m_{k}^{\max}}F_{\beta_{k+1}}^{m_{k+1}^{\max}}\cdots F_{\beta_N}^{m_N^{\max}}v_\lambda\neq 0$.
\end{itemize}
The tuple $\underline{\mathbf m}^{\max}$ gives rise to a sequence of tuples in $\mathbb N^N$: for $k=1,\ldots, N$, set
$$
\underline{\mathbf m}_k^{\max}=(0,\ldots,0,m_{k}^{\max},m_{k+1}^{\max},\ldots,m_{N}^{\max}).
$$
\begin{lemma}\label{lemma23} 
For all $k=1,\ldots,N$, the following statements hold:
\begin{enumerate}
\item $m_{k}^{\max}=\langle \lambda,\alpha^\vee_{i_k}\rangle>0$ for all $k=1,\ldots,N$;
\item $F_{\beta_{k}}^{m_{k}^{\max}}\cdots F_{\beta_N}^{m_N^{\max}}v_\lambda$ is a weight
vector of weight $s_{\beta_k}\cdots s_{\beta_N}(\lambda)$;
\item  $\underline{\mathbf m}_k^{\max}$ is essential for $V(\lambda)$.
\end{enumerate}
\end{lemma}
\begin{proof}
The proofs of $(1)$ and $(2)$ are executed by descending induction. Note that for $k=N$,
$\beta_N$ is a simple root. Since $\lambda$ is a regular dominant weight, $\mathfrak{sl}_2$-theory implies
$(1)$ and $(2)$.

Suppose now $k<N$ and the claims hold
for $k+1$. Then $F_{\beta_{k+1}}^{m_{k+1}^{\max}}\cdots F_{\beta_N}^{m_N^{\max}}v_\lambda$ is an
extremal weight vector, hence is either a highest or a lowest weight vector for the subalgebra $\mathfrak{sl}_2(\beta_k)$ 
of $\mathfrak g$ generated by the root vectors $F_{\beta_k},E_{\beta_k}$ and $\beta^\vee_k$. Recall from  \eqref{decomplistb} that
$\beta_k=s_{i_N}\cdots s_{i_{k+1}}(\alpha_{i_k})$. Since
 $$
 \begin{array}{rcl}
 \langle s_{\beta_{k+1}}\cdots s_{\beta_N}(\lambda),\beta_k^\vee\rangle
 &=& \langle s_{i_N}\cdots s_{i_{k+1}}(\lambda), s_{i_N}\cdots s_{i_{k+1}}(\alpha_{i_k}^\vee)\rangle\\
 &=&\langle \lambda,\alpha_{i_k}^\vee\rangle
 \end{array}
 $$
is strictly positive, $F_{\beta_{k+1}}^{m_{k+1}^{\max}}\cdots F_{\beta_N}^{m_N^{\max}}v_\lambda$ is a highest weight vector
for $\mathfrak{sl}_2(\beta_k)$. It follows that $m_{k}^{\max}=\langle \lambda,\alpha^\vee_{i_k}\rangle>0$,
and $F_{\beta_{k}}^{m_{k}^{\max}}\cdots F_{\beta_N}^{m_N^{\max}}v_\lambda$ is a weight
vector of weight $s_{\beta_k}\cdots s_{\beta_N}(\lambda)$, which proves the claim by induction.

To prove $(3)$, notice that for a tuple $\underline{\mathbf m}<_{or} \underline{\mathbf m}_k^{\max}$
one has only two possibilities: either $F^{ \underline{\mathbf m}}v_\lambda=0$, or the weight
of $F^{ \underline{\mathbf m}}v_\lambda$ is different from the weight of $F^{ \underline{\mathbf m}_k^{\max}}v_\lambda$. Indeed, $F^{ \underline{\mathbf m}}v_\lambda\neq 0$ and $\underline{\mathbf m}<_{or} \underline{\mathbf m}_k^{\max}$ imply that the last $N-k+1$ coordinates of $\mathbf{m}$ and $\underline{\mathbf m}_k^{\max}$ coincide. Now $m_1^{\max}=\cdots=m_{k-1}^{\max}=0$ and some of the first $k-1$ entries of $\underline{\mathbf m}$ are positive (as $\underline{\mathbf m}<_{or} \underline{\mathbf m}_k^{\max}$), so the weights of $F^{ \underline{\mathbf m}_k^{\max}}v_\lambda$ and $F^{ \underline{\mathbf m}}v_\lambda$ differ by a positive linear combination of positive weights.
\par
This implies that $F^{ \underline{\mathbf m}_k^{\max}}v_\lambda$ is linearly independent of the $F^{ \underline{\mathbf m}}v_\lambda$
with $\underline{\mathbf m}<_{or} \underline{\mathbf m}_k^{\max}$, and hence 
$\underline{\mathbf m}_k^{\max}$ is essential for $V(\lambda)$. 

\end{proof}

Denote by $\mathbf{e}_{i,N}\in\mathbb R^N$ the element $\mathbf{e}_{i,N}=\sum_{j=i}^N \mathbf{e}_{j}\in\mathbb R^N$. 
Recall that $\rho$ is the sum of all fundamental weights. Lemma~\ref{lemma23} implies that the tuples $\mathbf{e}_{i,N}$
are elements of $\es(\rho)\subset \Delta_\rho$. Set
$$
\mathfrak S_{\rho}^N=\text{Convex hull of\ }\{0,\mathbf{e}_{1,N} ,\mathbf{e}_{2,N}\ldots,\mathbf{e}_{N,N}\}\subseteq \Delta_\rho\subset \mathbb R^N.
$$
The polytope $\mathfrak S_{\rho}^N$ is obviously unimodularly equivalent to the closure of $\mathfrak S^N(1)$. 

\begin{theorem}\label{main1}
Let $\lambda$ be a regular dominant integral weight and $k=\min\{\langle\lambda,\beta^\vee\rangle;\beta\in \Phi^+\}$.
Fix a reduced decomposition $w_0=s_{i_1}\cdots s_{i_N}$. Let $\Gamma_\lambda$ be the associated
finitely generated monoid and denote by $\Delta_\lambda\subset \mathbb R^N$ the associated Newton-Okounkov polytope. Then  
$k\mathfrak S_{\rho}^N\subset \Delta_\lambda$.
\end{theorem}
\begin{proof}   
Recall that $\lambda=k\rho+\mu$ for some dominant integral  weight $\mu$, which implies that $\es(k\rho)\subset \es(\lambda)=\Gamma_{\lambda}(1)$
by Theorem~\ref{global}. Now Lemma~\ref{lemma23} implies that $k\mathbf{e}_{1,N} ,k\mathbf{e}_{2,N}\ldots,k\mathbf{e}_{N,N}\in \es(k\rho)$
and hence:
$$
\{0, k\mathbf{e}_{1,N} ,k\mathbf{e}_{2,N}\ldots,k\mathbf{e}_{N,N}\}\subset \es(\lambda)=\Gamma_{\lambda}(1)\subset \Delta_\lambda.
$$
The convexity of  $\Delta_\lambda$ implies that the convex hull of these points is contained in $\Delta_\lambda$,
and hence $k\mathfrak S_{\rho}^N\subseteq \Delta_\lambda$. 
\end{proof}

\subsection{A special simplex in $\Delta_\lambda$ from cominuscule telescopes}\label{cominuscule telescopes}
\subsubsection{A tower of Levi subalgebras}
Another approach to construct Newton-Okounkov bodies  $\Delta_\lambda=\Delta_\lambda(\underline{\beta})$ containing particular simplices uses a tower of Levi subalgebras
and is in this sense in the spirit of Thimm's trick \cite{T}.
We assume in the following that $\lambda$ is a {\it regular dominant integral weight}. 
Let  $\alpha_1,\alpha_2,\ldots,\alpha_n$ be an enumeration of the simple roots of $\mathfrak g$, and $\varpi_1,\ldots,\varpi_n$ be the associated fundamental weights.
Let $\mathfrak l_j\subset \mathfrak g$ be the Levi subalgebra 
associated to the subset of simple roots $\{\alpha_1,\alpha_2,\ldots,\alpha_j\}$.
The enumeration induces an increasing sequence of Levi subalgebras
\begin{equation}
\label{Levi}
\mathfrak l_0=\mathfrak t\subset \mathfrak l_1\subset\mathfrak l_2\subset\ldots  \subset\mathfrak l_n=\mathfrak g.
\end{equation}
Set $\mathfrak n^-_j=\mathfrak n^-\cap \mathfrak l_j$, then we have an induced sequence of inclusions:
\begin{equation}
\label{nilLevi}
\mathfrak n^-_0=0\subset \mathfrak n^-_1\subset\mathfrak n^-_2\subset\ldots  \subset\mathfrak n^-_n=\mathfrak n^-.
\end{equation} 
Let $\Phi(\mathfrak l_j)^+$ be the set of positive roots of the Levi subalgebra $\mathfrak l_j$.
For $j=1,\ldots,n$, set $\Phi_j^+=\Phi(\mathfrak l_j)^+\setminus\Phi(\mathfrak l_{j-1})^+$, then \eqref{nilLevi} induces a partition of the set of positive roots $\Phi^+$:
\begin{equation}
\label{roots}
\Phi^+=\Phi_1^+ \sqcup \Phi_2^+ \sqcup \cdots \sqcup  \Phi_n^+.
\end{equation}
We fix now an enumeration of the set of positive roots which is compatible with the partition above. More precisely,
there exist $N= i_1>i_{2} >\ldots > i_{n-1}>i_n=1$ such that  for any $1\leq j\leq N$:
\begin{equation}\label{enumlevi}
\Phi_1^+ \sqcup \Phi_2^+ \sqcup \cdots \sqcup  \Phi_j^+=\{\beta_{i_j},\ldots,\beta_N\}.
\end{equation}
In the following, we take the PBW basis of the enveloping algebra $U(\mathfrak n^-)$ 
with respect to this enumeration. As in the sections before,
from now on we fix as a total order the {\it opposite right lexicographic order} on the tuples and on the monomials.
\subsubsection{Levi subalgebras and essential monomials for $V(\varpi_j)$}
We will use the sequence in \eqref{Levi} to provide an inductive procedure to construct essential monomials. 
The fundamental weight $\varpi_j$ can be viewed as a fundamental weight
for $\mathfrak g$, as well as a fundamental weight for the Levi subalgebra $\mathfrak l_j$. We write
$V(\varpi_j)$ for the irreducible $\mathfrak g$-representation and $\mathcal V(\varpi_j)$ for the irreducible
$\mathfrak l_j$-representation. By fixing a highest weight vector $v_{\varpi_j}$ in $V(\varpi_j)$, we may identify
$\mathcal V(\varpi_j)$ with the cyclic $\mathfrak l_j$-submodule of $V(\varpi_j)$ generated by the highest weight vector 
$v_{\varpi_j}$.

The filtration of $U(\mathfrak n^-)$ defined in Section~\ref{afiltration} also makes  sense for the enveloping algebra 
$U(\mathfrak n_j^-)$. To be able to compare the filtrations of the algebras and the induced filtrations 
on the representations $V(\varpi_j)$ and $\mathcal V(\varpi_j)$, recall that
there exists a number $i_j$ (see \eqref{enumlevi}) such that $\Phi(\mathfrak l_j)^+=\{\beta_{i_j},\ldots,\beta_N\}$. In the following we take the
PBW-basis of $U(\mathfrak n_j^-)$ with respect to this enumeration, which are monomials of the form
$F_{\beta_{i_j}}^{m_{i_j}}\cdots F_{\beta_N}^{m_N}$.
We can thus identify
the enveloping algebra $U(\mathfrak n_j^-)$ with the linear span of all ordered monomials $F^{\mathbf{\underline m}}$
in  $U(\mathfrak n^-)$ such that 
\begin{equation}\label{leviindex}
m_1=m_2=\ldots=m_{i_j-1}=0.
\end{equation}

If $\underline{\mathbf m}$ is as in \eqref{leviindex}, the monomial $F^{\mathbf{\underline m}}$ is an element of both $U(\mathfrak n^-_j)$
and $U(\mathfrak n^-)$, giving two notions of being essential: ${\mathbf{\underline m}}$ can be essential for either
the $\mathfrak l_j$-representation $\mathcal V(\varpi_j)$, or the $\mathfrak g$-representation 
$V(\varpi_j)$. 

\begin{lemma}\label{induc1}
If $\underline{\mathbf m}$ is as in \eqref{leviindex}, then ${\mathbf{\underline m}}$
is   essential   for the $\mathfrak l_j$-representation $\mathcal V(\varpi_j)$ if and only if it is so 
for the $\mathfrak g$-representation $V(\varpi_j)$.
\end{lemma}
\begin{proof}
Suppose that the tuples ${\mathbf{\underline a}^1},\ldots,{\mathbf{\underline a}^r}\in\mathbb N^N$ are essential
for $\varpi_j$, satisfying $\mathbf{\underline a}^k < \mathbf{\underline m}$ for $k=1,\ldots,r$ and
\begin{equation}
\label{sum}
F^{\mathbf{\underline m}} v_{\varpi_j}=\sum_{k=1}^r c_kF^{\mathbf{\underline a}^k} v_{\varpi_j}.
\end{equation}
All positive roots such that the root vector $F_\gamma$ occurs in  $F^{\mathbf{\underline m}}$ are
elements in $\Phi(\mathfrak l_j)^+$ by assumption, and hence these roots are linear combinations of the simple roots $\alpha_1,\ldots,\alpha_j$. For weight reasons, all positive roots $\gamma$ such that the root vector $F_\gamma$ occurs in one of the monomials
on the right hand side in \eqref{sum} must also be linear combinations of the simple roots $\alpha_1,\ldots,\alpha_j$. 
Hence the monomials occurring on the right hand side are elements of the enveloping algebra of $\mathfrak l_j$.
This implies that in this special situation being essential for $\mathcal V(\varpi_j)$ is equivalent to being essential for $V(\varpi_j)$.
\end{proof}

\subsubsection{Some special essential monomials for cominuscule weights}
\begin{lemma}\label{special1}
If $\varpi_j$ is a cominuscule weight for $\mathfrak l_j$, then all root vectors $F_\beta$, $\beta\in  \Phi^+_j$,
are essential monomials for the $\mathfrak g$-representation $V(\varpi_j)$.
\end{lemma}

\begin{proof}
By Lemma~\ref{induc1} it is sufficient to prove that  $F_\beta$ is an essential monomial for the 
$\mathfrak l_j$-representation $\mathcal V(\varpi_j)$. Let $\mathbf{\underline m}_\beta\in \mathbb N^N$
be such that  $F_\beta=F^{\mathbf{\underline m}_\beta}$. Suppose there exist monomials 
$F^{\mathbf{\underline a}^1},\ldots,F^{\mathbf{\underline a}^r}\in U(\mathfrak n_j^-)$ such that $\mathbf{\underline a}^k < \mathbf{\underline m}_\beta$
for $k=1,\ldots,r$, and
\begin{equation}
\label{sum2}
F^{\mathbf{\underline m}_\beta} v_{\varpi_j}=\sum_{k=1}^r c_k F^{\mathbf{\underline a}^k} v_{\varpi_j}.
\end{equation}
For a monomial $F^{\mathbf{\underline a}^k}$ with nonzero coefficient $c_k$ and a positive root $\gamma_i\in\Phi^+(\mathfrak l_j)$,
let $a_i$ be the exponent of $F_{\gamma_i}$ in the monomial. Comparing the weights on both sides of \eqref{sum2}, the
sum $\sum_{\ell=i_j}^N a_\ell\gamma_\ell$ is equal to $\beta$. Since $\varpi_j$ is assumed to be cominuscule,
the coefficient of $\alpha_j$ in the expression of $\beta$ as a sum of simple roots is equal to one. 
It follows, that at least for one of the roots one has that $a_i>0$ but $\gamma_i\not\in \Phi^+_j$. This implies
$\gamma_i\in \Phi^+(\mathfrak l_{j-1})$, and, because of the fixed order in which the monomials are written, we have
$F^{\mathbf{\underline a}^k}v_{\varpi_j}=0$.
As a consequence, a linear dependence relation as in \eqref{sum2} is not possible.
\end{proof}

\begin{theorem}\label{Thm:cominuscule}
Let $\lambda$ be a regular dominant integral weight and $k=\min\{\langle\lambda,\beta^\vee\rangle;\beta\in \Phi^+\}.$
If $\mathfrak g$ is not of type $\tt G_2, \tt F_4$ or $\tt E_8$, then there exists an enumeration of the positive
roots such that $\Gamma_\lambda$ is finitely generated and $\mathfrak S^N(k)\subset \Delta_\lambda$.
\end{theorem}
\begin{proof}
By \cite{FFL}, $\Gamma_\lambda=\Es(\lambda)$ is finitely generated if, for example, one can show that there exists
a reduced decomposition of $w_0$ such that the induced ordering has the properties demanded in \eqref{enumlevi}.

\par
If $\mathfrak g$ is not of type $\tt G_2, \tt F_4$ or $\tt E_8$, then
there exists an enumeration $\{\alpha_1,\ldots,\alpha_n\}$ of the simple roots such that for all $j=1,\ldots,n$, the 
fundamental weight $\varpi_j$ is cominuscule for the Levi subalgebra $\mathfrak l_j$, \emph{i.e.,} $\langle\varpi_j,\theta_j^\vee\rangle=1$ for the highest root $\theta_j$ in $\mathfrak{l}_j$.
For $\mathfrak g$ of type $\tt A_{n}$, any enumeration of the simple roots has this property. 
For $\mathfrak g$ of type $\tt C_{n}$
or $\tt D_n$, one takes the standard enumeration as in \cite{Bo}. For $\mathfrak g$ of type $\tt B_{n}$
take $\alpha_1$ to be the only short root among the simple roots, then add successively the only
simple root which is connected in the Dynkin diagram to the already chosen ones. 
For $\mathfrak g$ of type $\tt E_6$ take first the enumeration of the simple roots as in the  $\tt D_5$ case above,
and then add the only missing simple root as the sixth one. For $\mathfrak g$ of type $\tt E_7$ take first the enumeration 
of the simple roots as in the  $\tt E_6$ case above, and then add the only missing simple root as the seventh one. 
\par
To get an enumeration of the positive roots induced by a reduced decomposition
and also satisfying \eqref{enumlevi}, let $w_0$ be the longest word in the Weyl group $W$ of $G$, $w_0^j$ be the longest word in the Weyl group of $\mathfrak l_{j}$, and 
$$
\tau_j\equiv w_0^j\ \ (\bmod\ W_{j-1})
$$
be a minimal representative in $W_j$ of the class of $w_0^j$ in $W_{j-1}\backslash W_j$.
Note that the block decomposition $w^j_0=\tau_1\tau_2\cdots\tau_j$ is such that the lengths add up,
\emph{i.e.}, $\ell(w^j_0)=\sum_{s=1}^j \ell(\tau_s)$ for all $j=1,\ldots,n$.
We fix a reduced decomposition of the longest word $w_0$ which is
compatible with this block decomposition, \emph{i.e.,} the decomposition of the form:
\begin{equation}
\label{decomp}
w_0=
 \underbrace{s_1}_{\tau_1}
 \underbrace{s_2 s_1\cdots}_{\tau_2}\ \cdots
 \ \underbrace{s_{j_1}\cdots s_{j_r}}_{\tau_j}\ \cdots
\ \underbrace{s_{n_1}\cdots s_{n_t}}_{\tau_n}.
\end{equation}
It is now easy to see that the enumeration of the elements of $\Phi^+$ induced by this reduced decomposition
(see \eqref{decomplist}, here $L=T$) has the properties described in \eqref{enumlevi}. 

Lemma~\ref{special1} implies that with respect to this enumeration,
every root vector is essential for some fundamental representation. Now
the same arguments as in the proof of Theorem~\ref{main1} imply
that for any $i=1,\ldots,N$, $\mathbf{e}_i$ is essential for $V(\rho)$. By the same arguments as above, $k$ is maximal such that $\lambda=k\rho +\mu$ for some dominant weight $\mu$,
so $\es(k\rho)\subset \es(\lambda)$ and hence 
$$
\{0,k\mathbf{e}_1,k\mathbf{e}_2,\ldots,k\mathbf{e}_N\}\subset \Delta_\lambda.
$$
Since $\Delta_\lambda$ is convex, it follows that $\mathfrak S^N(k)\subset \Delta_\lambda$.
\end{proof}
\begin{remark}
Recall that by our construction the Newton-Okounkov body is contained in the positive octant, thus 
in this case $\mathfrak S^N(k)$ is the intersection of $\Delta_\lambda$ with an affine half-space. 
Corollary \ref{cor action on smooth pts} implies hence that the supremum appearing in the definition of Gromov width is attained: there exists a symplectic embedding of a ball of capacity $k$.
\end{remark}

We give several examples on the construction in Theorem \ref{Thm:cominuscule}.

\begin{example}
We fix the following notations: for $1\leq i\leq j\leq n$, $\alpha_{i,j}=\alpha_i+\ldots+\alpha_j$; when $\mathfrak{g}$ is of type $\tt B_n$, for $1\leq i<j\leq n$, $\alpha_{i,\overline{j}}=\alpha_i+\ldots+\alpha_{n}+\alpha_n+\ldots+\alpha_j$; when $\mathfrak{g}$ is of type $\tt C_n$, for $1\leq i\leq j\leq n$, $\alpha_{i,\overline{j}}=\alpha_i+\ldots+\alpha_{n-1}+\alpha_{n}+\alpha_{n-1}+\ldots+\alpha_j$.
\begin{enumerate}
\item For $\mathfrak{g}$ of type $\tt A_n$, by construction, we may consider the enumeration of positive roots arising from the following inclusions of Levi subalgebras:
$$\mathfrak{sl}_2\subset \mathfrak{sl}_3\subset\ldots\subset \mathfrak{sl}_n\subset\mathfrak{sl}_{n+1}.$$
For example, when $n=3$, the enumeration can be chosen as 
$$(\alpha_{3,3},\alpha_{2,3},\alpha_{1,3},\underbrace{\alpha_{2,2},\alpha_{1,2},\alpha_{1,1}}_{\tt A_2}).$$
\item For $\mathfrak{g}$ of type $\tt B_n$, we consider the enumeration of positive roots arising from the following inclusions of Levi subalgebras:
$$\mathfrak{sl}_2\subset \mathfrak{so}_5\subset\mathfrak{so}_7\subset\ldots\subset \mathfrak{so}_{2n-1}\subset\mathfrak{so}_{2n+1}.$$
For instance, when $n=3$, the enumeration can be chosen as
$$(\alpha_{1,1},\alpha_{1,2},\alpha_{1,3},\alpha_{1,\overline{3}},\alpha_{1,\overline{2}},\underbrace{\alpha_{2,2},\alpha_{2,3},\alpha_{2,\overline{3}},\alpha_{3,3}}_{\tt B_2}).$$
\item For $\mathfrak{g}$ of type $\tt C_n$, we consider the enumeration of positive roots arising from the following inclusions of Levi subalgebras:
$$\mathfrak{sl}_2\subset \mathfrak{sl}_3\subset\ldots\subset \mathfrak{sl}_{n}\subset\mathfrak{sp}_{2n}.$$
For example, when $n=3$, the enumeration can be chosen as 
$$(\alpha_{3,3},\alpha_{2,3},\alpha_{2,\overline{2}},\alpha_{1,3},\alpha_{1,\overline{2}},\alpha_{1,\overline{1}},\underbrace{\alpha_{2,2},\alpha_{1,2},\alpha_{1,1}}_{\tt A_2}).$$
\item For $\mathfrak{g}$ of type $\tt D_n$, there are different ways to obtain enumerations of positive roots from inclusions of Levi subalgebras. For example,
$$\mathfrak{sl}_2\subset \mathfrak{sl}_3\subset \mathfrak{sl}_4\subset \mathfrak{so}_8\subset\ldots\subset \mathfrak{so}_{2n-2}\subset\mathfrak{so}_{2n}\ \ or\ \ \mathfrak{sl}_2\subset \mathfrak{sl}_3\subset\ldots\subset \mathfrak{sl}_{n}\subset\mathfrak{so}_{2n}.$$

\end{enumerate}
\end{example}

\end{document}